\documentclass[11pt,leqno]{article}
\usepackage{amsthm,amsfonts,amssymb,amsmath,oldgerm}
\usepackage{epsfig}
\numberwithin{equation}{section}
\usepackage[thinlines]{easybmat}

\def\eps{\varepsilon }

\def\eps{\varepsilon}

\newcommand\br{\begin{remark}}
\newcommand\er{\end{remark}}
\newcommand\bp{\begin{pmatrix}}
\newcommand\ep{\end{pmatrix}}
\newcommand\be{\begin{equation}}
\newcommand\ee{\end{equation}}
\newcommand\ba{\begin{equation}\begin{aligned}}
\newcommand\ea{\end{aligned}\end{equation}}

\newcommand{\bap}{\begin{app}}
\newcommand{\eap}{\end{app}}
\newcommand{\begs}{\begin{exams}}
\newcommand{\eegs}{\end{exams}}
\newcommand{\beg}{\begin{example}}
\newcommand{\eeg}{\end{exaplem}}
\newcommand{\bpr}{\begin{proposition}}
\newcommand{\epr}{\end{proposition}}
\newcommand{\bt}{\begin{theorem}}
\newcommand{\et}{\end{theorem}}
\newcommand{\bc}{\begin{corollary}}
\newcommand{\ec}{\end{corollary}}
\newcommand{\bl}{\begin{lemma}}
\newcommand{\el}{\end{lemma}}
\newcommand{\bd}{\begin{definition}}
\newcommand{\ed}{\end{definition}}
\newcommand{\brs}{\begin{remarks}}
\newcommand{\ers}{\end{remarks}}

\newtheorem{theo}{Theorem}[section]

\newtheorem{exams}[theo]{Examples}

\numberwithin{equation}{section}
\newcommand{\MM}{{\mathbb M}}

\newtheorem{theorem}{Theorem}[section]
\newtheorem{proposition}[theorem]{Proposition}
\newtheorem{corollary}[theorem]{Corollary}
\newtheorem{lemma}[theorem]{Lemma}
\newtheorem{definition}[theorem]{Definition}

\newtheorem{example}[theorem]{Example}
\newtheorem{remark}[theorem]{Remark}

\newtheorem{thm}{Theorem}

\newcommand{\RM}{\mathbb{R}}

\newcommand{\CM}{\mathbb{C}}

\newcommand{\spec}{\operatorname{spec}}

\title{%On the Long-Wavelength Instability of Periodic GKdV Waves: Bloch Decompositions and Whitham  Theory
On the Modulation Equations and Stability of Periodic GKdV Waves via Bloch Decompositions
}

\author{\sc \small
Mathew A. Johnson\thanks{Indiana University, Bloomington, IN 47405;
matjohn@indiana.edu: Research of M.J. was partially supported by an NSF Postdoctoral Fellowship under NSF grant no. DMS-0902192.},
~
Kevin Zumbrun\thanks{Indiana University, Bloomington, IN 47405;
kzumbrun@indiana.edu:
Research of K.Z. was partially supported
under NSF grants no. DMS-0300487 and DMS-0801745.
 },
~\&~
Jared C. Bronski\thanks{University of Illinois at Urbana-Champaign, Urbana, IL 61801; jared@math.uiuc.edu: Research of J.C.B.
was partially supported under NSF grants no. NSF DMS 0807584.}}
\begin{document}

\maketitle

%%%%%%%%%%%%%%%%%%%%%%%%%%%%%%%%%%%%%%%%%%%%%%%%%%%%%%%%%%%%%%%%%%%%%%%%%%%%%%%%%%%%%%%%%%%%%

\begin{center}
\end{center}

%%%%%%%%%%%%%%%%%%%%%%%%%%%%%%%%%%%%%%%%%%%%%%%%%%%%%%%%%%%%%%%%%%%%%%%%%%%%%%%%%%%%%%%%%%%%%

\begin{abstract}
In this paper, we complement recent results of Bronski and Johnson and of Johnson and Zumbrun concerning the modulational
stability of spatially periodic traveling wave solutions of the generalized Korteweg-de Vries equation.
In this previous work it was shown by rigorous Evans function calculations that the formal slow modulation
approximation resulting in the Whitham system accurately describes the spectral stability to long wavelength
perturbations.  Here, we reproduce this result without reference to the Evans function by using direct Bloch-expansion
methods and spectral perturbation analysis.  This approach has the advantage of applying also in the more general
multi-periodic setting where no conveniently computable Evans function is yet devised.  In particular, we complement
the picture of modulational stability described by Bronski and Johnson by analyzing
the projectors onto the total eigenspace bifurcating from the origin in a neighborhood
of the origin and zero Floquet parameter.  %This procedure is non-trivial due to the presence of a Jordan block at the origin.
We show the resulting linear system is equivalent, to leading order and up to conjugation, to the Whitham system and
that, consequently, the characteristic polynomial of this system agrees (to leading order) with the linearized dispersion
relation derived through Evans function calculation.
\end{abstract}

%\pagestyle{myheadings}
%\thispagestyle{plain}
%\markboth{Mathew A. Johnson Kevin Zumbrun}{Transverse Instability of gKdV}

%\clearpage
%\tableofcontents
%\clearpage
%%%%%%%%%%%%%%%%%

\bigbreak
\section{Introduction}

In this paper, we consider traveling wave solutions of the Korteweg-de Vries equation
\begin{equation}
u_t=u_{xxx}+f(u)_x\label{eqn:gkdv}
\end{equation}
where $u$ is a scalar, $x,t\in\RM$, and $f\in C^2(\RM)$ is a suitable nonlinearity.  Such equations
arise in a variety of applications.  For example, in the case $f(u)=u^2$ equation \eqref{eqn:gkdv}
reduces to the well known KdV equation, which serves as a canonical model of unidirectional
weakly dispersive nonlinear wave propagation \cite{KdV,SH}.  Furthermore, the cases $f(u)=\beta u^3$, $\beta=\pm 1$, correspond
to the modified KdV equation which arises as a model for large amplitude internal waves in a density stratified medium, as well
as a model for Fermi--Pasta-Ulam lattices with bistable nonlinearity \cite{BM} \cite{BO}.  In each of these
two cases, the corresponding PDE can be realized as a compatibility condition for a particular Lax pair and hence
the corresponding Cauchy problem can (in principle) be completely solved via the famous inverse scattering
transform\footnote{More precisely, the Cauchy problem for equation \eqref{eqn:gkdv} can be solved via the inverse
scattering transform if and only if $f$ is a cubic polynomial.}.  However, there are a variety of applications
in which equations of form \eqref{eqn:gkdv} arise which are not completely integrable and hence the inverse scattering
transform can not be applied.  For example, in plasma physics equations of the form \eqref{eqn:gkdv}
arise with a wide variety of power-law nonlinearities depending on particular physical considerations \cite{KD} \cite{M} \cite{MS}.
Thus, in order to accommodate as many applications as possible, the methods employed in this paper will not rely
on complete integrability of the PDE \eqref{eqn:gkdv}.  Instead, we will make use of the integrability of the ordinary
differential equation governing the traveling wave profiles.  This ODE is always Hamiltonian, regardless
of the integrability of the corresponding PDE.

It is well known that \eqref{eqn:gkdv} admits traveling wave solutions of the form
\begin{equation}
u(x,t)=u_c(x+ct)\label{eqn:uc}
\end{equation}
for $c>0$.  Historically, there has been much interest in the stability of
traveling waves of the form (1.2) where the profile $u_c$ decays exponentially
to zero as its argument becomes unbounded. Such waves were originally discovered by
Scott Russell in the case of the KdV where the traveling wave is termed a soliton.  While
\eqref{eqn:gkdv} does not in general possess soliton solutions, which requires complete integrability
of the governing PDE, it nonetheless admits exponentially decaying traveling wave profiles known
as solitary waves.  Moreover, the stability of such solitary waves is well understood and dates
back to the pioneering work of Benjamin \cite{B}, which was then further developed by Bona, Frillakis, Grillakis et. al.,
Bona et. al, Pego and Weinstein, Weinstein, and many others.

The subject of this paper, however, is the stability of the traveling wave solutions of \eqref{eqn:gkdv} which
are periodic, i.e. solutions of the form \eqref{eqn:uc} such that $u_c$ is $T$-periodic.  Unlike the
corresponding solitary wave theory, the stability of such periodic traveling waves is much less
understood.  Existing results in general fall into two categories: spectral stability with
respect to localized or periodic perturbations \cite{BrJ,BD}, or nonlinear (orbital) stability with
respect to periodic perturbations \cite{BrJK,J1,DK}.  In the present manuscript, we will consider
spectral stability to long-wavelength localized perturbations, i.e. spectral stability to
slow modulations of the underlying wave.  Such an analysis was recently carried out by Bronski
and Johnson in \cite{BrJ} by Evans function techniques.  In particular, by studying the Evans function
in a neighborhood of the origin a modulational stability index was derived which uses geometric information
concerning the underlying wave to determine the local structure of the $L^2$ spectrum of the linearized
operator at the origin.  This sign of this index was then evaluated in several examples in the subsequent work
of Bronski, Johnson, and Kapitula \cite{BrJK}.

Our approach is essentially different from \cite{BrJ}.  In particular, we study the modulational stability by
using a direct Bloch-expansion of the linearized eigenvalue problem as opposed to more familiar Evans function techniques.
This strategy offers several novel advantages.  First off, the Evans function techniques of \cite{BrJ} do not extend
to multiple dimensions in a straight forward way.  In particular, these techniques seem not to apply to study the stability
of multiply periodic solutions, that is, solutions which are periodic in more than one linearly independent
spatial directions.  Such equations are prevalent in the context of viscous conservation laws in multiple dimensions.
In contrast, the Bloch expansion approach of the current work generalizes to all dimensions with no difficulty.
Secondly, the low-frequency Evans function analysis
required to determine modulational instability is often quite tedious and difficult; see, for example,
the analysis in the viscous conservation law case in \cite{OZ1,OZ3,OZ4,Se1}.  In contrast, however, we find
that the Bloch-expansion methods of this paper are much more straight forward and reduces the problem
to elementary matrix algebra.  Using this approach, then, we reproduce the modulatioanl
stability results of Bronski and Johnson with out any specific mention or use of the Evans funcition.
While this direct approach has been greatly utilized in studying the stability
of small amplitude solutions of \eqref{eqn:gkdv} (see \cite{DK,GH,Ha,HK,HaLS} for example), we are unaware of
any previous applications to arbitrary amplitude solutions.

Given the above remarks, we choose to carry out our analysis first in the gKdV case
since, quite surprisingly, all calculations can be done very explicitly and can be readily
compared to the results of \cite{BrJ}.  Our hope is that the present analysis may serve
as a blueprint of how to consider the stability multiply periodic structures in more general, and difficult, equations
than simply the gKdV equation.  We should remark, however, that a slight variation of our approach
was attempted in the latter half of the analysis in \cite{BrJ}.  However, the analysis was quite cumbersome
requiring the asymptotic tracking of eigenfunctions which collapse to a Jordan block.  Our approach does not
require such delicate tracking and it is seen that we only need to asymptotically construct the projections
onto the total eigenspace, which requires only tracking of an appropriate basis not necessarily consisting
of eigenfunctions.

Finally, our approach connects in a very interesting way to the recent work of Johnson and Zumbrun \cite{JZ1}.
There, the authors studied the modulational stability of periodic gKdV waves in the context of
Whitham theory; a well developed (formal) physical theory for dealing with such stability problems.
The calculation proceeded by rescaling the governing PDE via the change of variables $(x, t)\mapsto (\eps x, \eps t)$
and then uses a WKB approximation of the solution to find a homogenized system which describes the mean behavior of
the resulting approximation.  In particular, it was found that a necessary condition
for the stability of such solutions is the hyperbolicity – i.e., local well-posedness – of the
resulting first order system of partial differential equations by demonstrating that the characteristic
polynomial of the linearized Whitham system accurately describes, to first order, the linearized dispersion
relation arising in the Evans function analysis in \cite{BrJ}.  As seen in the recent work of \cite{JZ4} and
\cite{JZ5}, there is a deeper analogy between the low-frequency linearized dispersion relation and the Whitham
averaged system at the structural level, suggesting a useful rescaling of the low-frequency perturbation problem.
Moreover, it was seen that the Whitham modulation equations can provide invaluable information concerning not only
the stability of nonlinear periodics, but also their long-time behavior.
It is this intuition that motivates our low-frequency analysis, and ultimately leads to our derivation
of the Whitham system for the gKdV using Bloch-wave expansions.  In particular, this
observation  provides us with a rigorous method of justifying the Whitham modulation equations which
is independent of Evans function techniques used in \cite{JZ1}.  As mentioned above, such a result is desirable
when trying to study the stability of multiply periodic waves where one does not have a useful
notion of an Evans function.

Therefore, the goal of this paper is to reproduce many of the results of \cite{BrJ} and \cite{JZ1}
in a method which is independent of the Evans function and can be applied to higher dimensional
and more general settings with little extra effort.  Of particular importance is the fact that we
can rigorously justify the Whitham modulation, a construct which exists in any dimension, using
more robust methods than the restrictive Evans function approach.  The plan for the paper is as follows.  In the next
section we discuss the basic properties of the periodic traveling wave solutions of the gKdV equation \eqref{eqn:gkdv},
including a parametrization which we will find useful in our studies.  We will then give a brief account of the
results and methods of the papers \cite{BrJ} and \cite{JZ1} concerning the Evans function and Whitham
theory approaches.  We will then begin our analysis by discussing the Bloch decomposition of the linaerized
operator about the periodic traveling wave.  In particular, we will show that the projection of the
resulting Bloch eigenvalue problem onto the total eigenspace is a three-by-three matrix which
is equivalent (up to similarity) to the Whitham system.  We will then end with some concluding remarks
and a discussion on consequences / open problems inferred from these results.

%Whitham: Average then linearize\\
%Bloch: Linearize then Average

\section{Preliminaries}

Throughout this paper, we are concerned with spatially periodic traveling waves of the gKdV equation \eqref{eqn:gkdv}.
To begin, we recall the basic properties of such solutions: for more information, see \cite{BrJ} or \cite{J1}.

Traveling wave solutions of the gKdV equation with wave speed $c>0$ are stationary solutions \eqref{eqn:gkdv} in the moving
coordinate frame $x+ct$ and whose profiles satisfy the traveling wave ordinary differential equation
\begin{equation}
u_{xxx}+f(u)_x-cu_x=0.\label{eqn:travelode}
\end{equation}
Clearly, this equation is Hamiltonian and can be reduced to quadrature through two integrations to the
nonlinear oscillator equation
\begin{equation}
\frac{u_x^2}{2}=E+au+\frac{c}{2}u^2-F(u)\label{eqn:quad2}
\end{equation}
where $F=f'$ and $F(0)=0$, and $a$ and $E$ are constants of integration.  Thus, the existence of periodic orbits
of \eqref{eqn:travelode} can be trivially verified through phase plane analysis: a necessary and sufficient condition
is that the effective potential energy $V(u;a,c):=F(u)-\frac{c}{2}u^2-au$ have a local minimum.  It follows then that
the traveling wave solutions of \eqref{eqn:gkdv} form, up to translation, a three parameter family of solutions
which can be parameterized by the constants $a$, $E$, and $c$.  In particular, on an open
(not necessarily connected) subset $\mathcal{D}$ of $\RM^3=\{(a,E,c):a,E,c\in\RM\}$ in which
the solutions to \eqref{eqn:travelode} are periodic with period $T=T(a,E,c)$: the boundary of these sets
corresponds to homoclinic/heteroclinic orbits (the solitary waves) and equilibrium solutions.

Moreover, we make the standard assumption that there exists simple roots $u_{\pm}$ with $u_{-}<u_{+}$
such that $E=V(u_{\pm};a,c)$ and $E>V(u;a,c)$ for all $u\in(u_{-},u_{+})$.  It follows
then that the roots $u_{\pm}$ are $C^1$ functions of the traveling wave parameters $a$, $E$, and $c$
and that, without loss of generality, $u(0)=u_{-}$.  Under these assumptions the functions $u_{\pm}$
are square root branch points of the function $\sqrt{E-V(u;a,c)}$.  In particular, the period
of the corresponding periodic solution of \eqref{eqn:travelode} can be expressed through
the integral formula
\begin{equation}
T=T(a,E,c)=\frac{\sqrt{2}}{2}\oint_\Gamma \frac{du}{\sqrt{E-V(u;a,c)}},\label{period}
\end{equation}
where integration over $\Gamma$ represents a complete integration from $u_{-}$ to $u_+$, and then back to $u_-$ again: notice
however that the branch of the square root must be chosen appropriately in each direction.  Alternatively, you could interpret
$\Gamma$ as a loop (Jordan curve) in the complex plane which encloses a bounded set containing both $u_-$ and $u_+$.
By a standard procedure, the above integral can be regularized at the square root branch points and hence represents
a $C^1$ function of the traveling wave parameters on $\mathcal{D}$.
Similarly, the conserved quantities of the gKdV flow can be expressed as
\begin{align}
M&=M(a,E,c)=\frac{\sqrt{2}}{2}\oint_\Gamma \frac{u~du}{\sqrt{E-V(u;a,c)}}\label{mass}\\
P&=P(a,E,c)=\frac{\sqrt{2}}{2}\oint_\Gamma \frac{u^2~du}{\sqrt{E-V(u;a,c)}}\label{momentum}\\
H&=H(a,E,c)=\frac{\sqrt{2}}{2}\oint_{\Gamma}\frac{E-V(u;a,c)-F(u)}{\sqrt{E-V(u;a,c)}}\label{ham}
\end{align}
where again these quantities, representing mass, momentum, and Hamiltonian, respectively, are finite and $C^1$ functions on $\mathcal{D}$.
As seen in \cite{BrJ,BrJK,J1}, the gradients of the functions $T,M,P:\mathcal{D}\to\RM$ play an important role
in the modulational stability analysis of periodic traveling waves of the gKdV equation \eqref{eqn:gkdv}.

\begin{remark}
Notice in the derivation of the gKdV \cite{BaMo}, the solution $u$ can represent either the horizontal
velocity of a wave profile, or the density of the wave. Thus, the function $M:\mathcal{D}\to\RM$ can properly be
interpreted as a ``mass" since it is the integral of the density over space. Similarly, the function $P:\mathcal{D}\to\RM$ can
be interpreted as a ``momentum" since it is the integral of the density times velocity over space.
\end{remark}

To assist with calculations involving gradients of the above conserved quantities considered as functions on $\mathcal{D}\subset\RM^3$,
we note the following useful identity.  The classical action (in the sense of action-angle variables) for the
profile equation \eqref{eqn:travelode} is given by
\[
K=\oint_\Gamma u_x~du=\sqrt{2}\oint_{\Gamma}\sqrt{E-V(u;a,c)}du
\]
where the contour $\Gamma$ is defined as above.  This provides a useful generating function for the conserved quantities of the gKdV
flow restricted to the manifold of periodic traveling waves.  Specifically, the classical action satisfies
\begin{equation}\label{eqn:gradk}
T=\frac{\partial K}{\partial E},~~M=\frac{\partial K}{\partial a},~~P=2\frac{\partial K}{\partial c}
\end{equation}
as well as the identity
\[
K=H+aM+\frac{c}{2}P+ET.
\]
Together, these relationships imply the important relation
\[
E\nabla_{a,E,c}T+a\nabla_{a,E,c}M+\frac{c}{2}\nabla_{a,E,c}P+\nabla_{a,E,c}H=0.
\]
So long as $E\neq 0$ then, gradients of the period, which is not itself conserved, can be
interchanged with gradientes of the genuine conserved quantities of the gKdV flow.
As a result, all gradients and geometric conditions involved in the results in this paper
can be expressed \emph{completely} in terms of the gradients of the conserved quantities of the gKdV flow, which
seems to be desired from a physical point of view.  However, as the quantities $T$, $M$, and $P$ arise most naturally
in the analysis, we shall state our results in terms of these quantities alone.

We now discuss our main assumptions concerning the parametrization of the family of periodic traveling wave
solutions of \eqref{eqn:gkdv}.  To begin, we assume throughout this paper the period is not at a critical
point in the energy, i.e. that $T_E\neq 0$.  In other words, we assume that the period serves as a good local
coordinate for nearby waves on $\mathcal{D}$ with fixed wave speed $c>0$ and parameter $a$.  As seen
in \cite{BrJK,J1}, such an assumption is natural from the viewpoint of nonlinear stability.  Moreover, we
assume the period and mass
provide good local coordinates for the periodic traveling waves of fixed wave speed $c>0$.  More precisely,
given $(a_0,E_0,c_0)\in\mathcal{D}$ with $c_0>0$, we assume the map
\[
(a,E)\mapsto\left(T(a,E,c_0),M(a,E,c_0)\right)
\]
have a unique $C^1$ inverse in a neighborhood of $(a_0,E_0)\in\RM^2$, which is clearly equivalent
with the non-vanishing of the Jacobian
\[
\det\left(\frac{\partial(T,M)}{\partial(a,E)}\right)
\]
at the point $(a_0,E_0,c_0)$.  As such Jacobians will be prevalent throughout our analysis, for notational
simplicity we introduce the following Poisson bracket notation for two-by-two Jacobians
\[
\{f,g\}_{x,y}:=\det\left(\frac{\partial(f,g)}{\partial(x,y)}\right)
\]
and the corresponding notation $\{f,g,h\}_{x,y,z}$ for three-by-three Jacobians.  Finally, we assume
that the period, mass, and momentum provide good local coordinates for nearby periodic traveling
wave solutions of the gKdV.  That is, given $(a_0,E_0,c_0)\in\mathcal{D}$ with $c_0>0$, we assume
that the Jacobian $\{T,M,P\}_{a,E,c}$ is non-zero.  While these re-parametrization conditions
may seem obscure, the non-vanishing of these Jacobians has been seen to be essential in both
the spectral and non-linear stability analysis of periodic gKdV waves in \cite{BrJ,J1,BrJK}.  In particular,
these Jacobians have been computed in \cite{BrJK} for several power-law nonlinearities and, in the cases considered,
has been shown to be generically non-zero.  Moreover, such a non-degeneracy condition should not be surprising as
a similar condition is often enforced in the stability theory of solitary waves (see \cite{Bo,B,PW}).

Now, fix a $(a_0,E_0,c_0)\in\mathcal{D}$.  Then the stability of the corresponding periodic traveling wave solution
may be studied directly by linearizing the PDE
\begin{equation}\label{eqn:travelpde}
u_t=u_{xxx}+f(u)_x-cu_x
\end{equation}
about the stationary solution $u(\cdot,a_0,E_0,c_0)$ and studying the $L^2(\RM)$ spectrum of the associated
linearized operator
\[
\partial_x\mathcal{L}[u]:=\partial_x\left(-\partial_x^2-f'(u)+c\right).
\]
As the coefficients of $\mathcal{L}[u]$ are $T$-periodic, Floquet theory implies the $L^2$ spectrum is purely continuous
and corresponds to the union of the $L^\infty$ eigenvalues corresponding to considering the linearized operator
with periodic boundary conditions $v(x+T)=e^{i\kappa}v(x)$ for all $x\in\RM$, where $\kappa\in[-\pi,\pi]$ is referred to as the Floquet
exponent and is uniquely defined mod $2\pi$.  In particular, it follows that $\mu\in\spec(\partial_x\mathcal{L}[u])$ if and only
if $\partial_x\mathcal{L}[u]$ has a bounded eigenfunction $v$ satisfying $v(x+T)=e^{i\kappa}v(x)$ for some $\kappa\in\RM$.
More precisely, writing the spectral problem $\partial_x\mathcal{L}[u]v=\mu v$ as a first order system
\begin{equation}\label{eqn:1order}
Y'(x)=H(x,\mu)Y(x),
\end{equation}
and defining the monodromy map $\MM(\mu):=\Phi(T)\Phi(0)^{-1}$, where $\Phi$ is a matrix solution of \eqref{eqn:1order}, it follows
that $\mu\in\CM$ belongs to the $L^2$ spectrum of $\partial_x\mathcal{L}[u]$ if and only if the periodic Evans function
\[
D(\mu,\kappa):=\det\left(\MM(\mu)-e^{i\kappa}I\right)
\]
vanishes for some $\kappa\in\RM$.

When studying the modulational stability of the stationary periodic solution $u(\cdot;a_0,E_0,c_0)$ it suffices
to study the zero set of the Evans function at low frequencies, i.e. seek solutions of $D(\mu,\kappa)=0$ for $|(\mu,\kappa)|\ll 1$.
Indeed, notice that the low-frequency expansion $\mu(\kappa)$ for $(\mu,\kappa)$ near $(0,0)$ may be expected
to determine the long-time behavior and can be derived by the lowest order terms of the Evans function
in a neighborhood of $(0,0)$.  As a first step in expanding $D$, notice by translation invariance of \eqref{eqn:travelpde} it follows that
\[
\partial_x\mathcal{L}[u]u_x=0,
\]
that is, $u_x$ is in the right $T$-periodic kernel of the $\partial_x\mathcal{L}[u]$.  It immediately follows
that $D(0,0)=0$, and hence to determine modulational stability we must find all solutions
of the form $(\mu(\kappa),\kappa)$ of the equation $D(\mu,\kappa)=0$ in a neighborhood of $\kappa=0$.  Using Noether's
theorem, or appropriately differentiating the integrated profile equation \eqref{eqn:quad2},
it follows that the functions $u_a$ and $u_E$ formally satisfy
\[
\partial_x\mathcal{L}[u]u_a=\partial_x\mathcal{L}[u]u_E=0,~~\partial_x\mathcal{L}[u]u_c=-u_x.
\]
However, these functions are not in general $T$-periodic due to the secular dependence
of the period on the parameters $a$, $E$, and $c$.  Nevertheless, one can take linear combinations of these functions
to form another $T$-periodic null-direction and a $T$-periodic function in a Jordan chain above the translation direction.
A tedious, but fairly straightforward, calculation (see \cite{BrJ}) now yields
\begin{equation}\label{eqn:evans}
D(\mu,\kappa)=\Delta(\mu,\kappa)+\mathcal{O}(|\mu|^4+\kappa^4),
\end{equation}
where $\Delta$ represents a homogeneous degree three polynomial of the variables $\mu$ and $\kappa$.  Defining
the projective coordinate $y=\frac{i\kappa}{\mu}$ in a neighborhood of $\mu=0$, it follows the modulatioanl
stability of the underlying periodic wave $u(\cdot;a_0,E_0,c_0)$ may then be determined by the
discriminant of the polynomial $R(y):=\mu^{-3}\Delta\left(1,-iy\right)$, which takes
the explicit form
\begin{equation}
R(y)=-y^3+\frac{y}{2}\left(\{T,P\}_{E,c}+2\{M,P\}_{a,E}\right)-\frac{1}{2}\{T,M,P\}_{a,E,c}.\label{delta1}
\end{equation}
Modulational stability then corresponds with $R$ having  three real roots, while the presence of root with non-zero imaginary
part implies modulational instability.

On the other hand, we may also study the stability of the solution $u(\cdot;a_0,E_0,c_0)$ through the formal
approach of Whitham theory.  Indeed, recalling the recent work of \cite{JZ1} we find upon rescaling \eqref{eqn:gkdv} by $(x,t)\mapsto (\eps x,\eps  t)$
and carrying out a formal WKB  expansion as $\eps\to 0$ a closed form system of three averaged, or homogenized, equations
of the form
\begin{equation}
\partial_t\left<M\omega,P\omega,\omega\right>(\dot{u})+\partial_x\left<F,G,H\right>(\dot{u})=0\label{eqn:whitham}
\end{equation}
where $\omega=T^{-1}$ and $\dot{u}$ represents a periodic traveling wave solution in the vicinity of the underlying (fixed) periodic
wave $u(\cdot;a_0,E_0,c_0)$.  The problem of stability of $u(\cdot;a_0,E_0,c_0)$ to long-wavelength perturbations may heuristically
expected to be related to the linearization of \eqref{eqn:whitham} about the constant solution $\dot{u}\equiv u(\cdot;a_0,E_0,c_0)$,
provided the WKB approximation is justifiable by stability considerations.  This leads one to the consideration of the homogeneous
degree three linearized dispersion relation
\begin{equation}
\hat{\Delta}(\mu,\kappa):=\det\left(\mu \frac{\partial(M\omega,P\omega,\omega)}{\partial(\dot{u})}
        -\frac{i\kappa}{T}\frac{\partial\left(F,G,H\right)}{\partial(\dot{u})}\right)(u(\cdot;a_0,E_0,c_0))=0\label{eqn:wpoly1}
\end{equation}
where $\mu$ corresponds to the Laplace frequency and $\kappa$ the Floquet exponent.

The main result of \cite{JZ1} was to demonstrate a direct relationship between the above approaches.  In particular, the following
theorem was proved.

\begin{thm}[\cite{JZ1}]\label{thm:JZ1}
Under the assumptions that the Whitham system \eqref{eqn:whitham} is of evolutionary type, i.e. $\frac{\partial(M\omega,P\omega,\omega)}{\partial(\dot{u})}$ is invertible
at $(a_0,E_0,c_0)\in\mathcal{D}$,
and that the matrix $\frac{\partial(\dot{u})}{\partial(a,E,c)}$ is invertible at $(a_0,E_0,c_0)\in\mathcal{D}$, we have that in a neighborhood
of $(\mu,\kappa)=(0,0)$ the asymptotic expansion
\[
D(\mu,\kappa)=\Gamma_0\hat{\Delta}(\mu,\kappa)+\mathcal{O}(|\mu|^4+\kappa^4)
\]
for some constant $\Gamma\neq 0$.
\end{thm}

\begin{remark}
The assumption in Theorem \ref{thm:JZ1} that $\frac{\partial(M\omega,P\omega,\omega)}{\partial(\dot{u})}$
is invertible forces the Whitham system \eqref{eqn:whitham} to be of evolutionary type.  Moreover, notice that the Whitham system
is inherently a relation of functions of the variable $\dot{u}$, while the Evans function calculations from \cite{BrJ} utilize
a parametrization of $\mathcal{P}$ by the parameters $(a,E,c)$.  In order to compare the two linearized dispersion relations $\Delta(\mu,\kappa)$
and $\hat{\Delta(\mu,\kappa)}$ then we must ensure that we may freely interchange between these variables, i.e. we must assume
that the matrix $\frac{\partial(\dot{u})}{\partial(a,E,c)}$ is invertible at the underlying periodic wave.
\end{remark}

That is, up to a constant, the dispersion relation \eqref{eqn:wpoly1} for the homogonized system \eqref{eqn:whitham} accurately describes
the low-frequency limit of the exact linearized dispersion relation described in \eqref{eqn:evans}.
As a result, Theorem \ref{thm:JZ1} may be regarded as a justification of the WKB expansion and the formal Whitham procedure
as applied to the gKdV equation.  The importance of this result stems from our discussion in the previous section: although
the Evans function techniques of \cite{BrJ} do not extend in a straight forward way to multiperiodic waves, the formal
Whitham procedure may still be carried out nonetheless.  However, it follows that we must find a more robust method
of study to justify the Whitham expansion in higher dimensional cases.  The purpose of this paper is to present
precisely such a method using Bloch-expansions of the linearized operator near zero-frequency.  Indeed, we will show
how this general method in the case of the gKdV equation may be used to easily rigorously reproduce the linearized dispersion relation
$\hat{\Delta}(\mu,\kappa)$ corresponding to the Whitham system as well as justifying the Whitham expansion beyond stability to the level of long-time behavior of the perturbation.

\section{Bloch Decompositions and Modulational Stability}

In this section, we detail the methods of our modulational stability analysis by
utilizing Bloch-wave decompositions of the linearized problem.  In particular, our goal is to rigorously justify the Whitham
expansion described in the previous section without reference to the Evans function.  To this end, %we define $\eps=\frac{i\kappa}{T}$ and
recall from Floquet theory that any bounded eigenfunction $v$ of $\partial_x\mathcal{L}[u]$ must satisfy
\[
v(x+T)=e^{i\gamma}v(x)
\]
for some $\gamma\in[-\pi,\pi]$.  Defining $\eps=\eps(\kappa)=\frac{i\kappa}{T}$, it follows that $v$ must be of the form
\[
v(x)=e^{\eps x}P(x)
\]
for some $\kappa\in[-\pi,\pi]$ and some $T$-periodic function $P$ which is an eigenfunction of the corresponding operator
\begin{equation}
L_{\eps}=e^{-\eps x}\partial_x\mathcal{L}[u]e^{\eps x}.\label{eqn:Leps}
\end{equation}
With this motivation, we introduce a one-parameter family of Bloch-operators $L_\eps$ defined
in \eqref{eqn:Leps} considered on the real Hilbert space $L^2_{\rm per}([0,T])$.  By standard results in Floquet theory, the spectrum of a
given operator $L_{\eps}$ is discrete consisting of point eigenvalues and satisfy
\[
\spec_{L^2(\RM)}\left(\partial_x\mathcal{L}[u]\right)=\bigcup_{\eps}\spec(L_{\eps})
\]
and hence the $L^2(\RM)$ spectrum of the linearized operator $\partial_x\mathcal{L}[u]$ can be parameterized
by the parameter $\eps$.  As a result, the above decomposition reduces the problem of determining the continuous spectrum
of the operator $\partial_x\mathcal{L}[u]$ to that of determining the discrete spectrum of the one-parameter family of operators
$L_{\eps}$.

As we are interested in the modulational stability of the solution $u$, we begin our analysis by studying the null-space of the
un-perturbed operator $L_0=\partial_x\mathcal{L}[u]$ acting on $L_{\rm per}([0,T])$.  We will see that under certain non-degeneracy
conditions $L_0$ has a two-dimensional kernel with a one-dimeisional Jordan chain.  It follows that the origin is a $T$-periodic eigenvalue
of $L_0$ with algebraic multiplicity three, and hence for $|\eps|\ll 1$, considering $L_{\eps}$ as a small perturbation of $L_0$, there will
be three eigenvalues bifurcating from the $\eps=0$ state.  To determine modulational stability then, we will determine conditions which
imply these bifurcating eigenvalues are confined to the imaginary axis.

%To begin, notice by the translation invariance of \eqref{eqn:gkdv} we have that
%\[
%L_0 u_x=0
%\]
%and hence $u_x$ is a $T$-periodic eigenfunction of the unperturbed linearized operator $L_0=\partial_x\mathcal{L}[u]$.  It
%follows that $\mu=0$ is an eigenvalue of the operator $L_0$ acting on the space $L_{\rm per}^2([0,T])$ of multiplicity
%at least one.  Moreover, differentiating \eqref{eqn:quad2} with respect to $a$, $E$, and $c$ implies the functions
%$u_a$ and $u_E$ satisfy
%\[
%\partial_x\mathcal{L}[u]u_a=\partial_x\mathcal{L}[u]u_E=0,~~\partial_x\mathcal{L}[u]u_c=-u_x.
%\]
%However, the functions $u_a$, $u_E$, and $u_c$ are not in general $T$-periodic due to the secular dependence
%of the period on the parameters $a$, $E$, and $c$.  However, one can take linear combinations of these functions
%to form another $T$-periodic null-direction and a $T$-periodic function in a Jordan chain.  We state this result
%in the next lemma: the relevant details are contained in \cite{BrJ,BrJK}.

To begin, we formalize the comments of the previous section concerning the structure of the generalized null-space of the unperturbed
linearized operator $L_0=\partial_x\mathcal{L}[u]$ by recalling the following lemma from \cite{BrJ,BrJK}.

\begin{lemma}[\cite{BrJ,BrJK}]\label{lem:Jordan}
Suppose that $u(x;a_0,E_0,c_0)$ is a $T$-periodic solution of the traveling wave ordinary differential equation
\eqref{eqn:travelode}, and that the Jacobian determinants $T_E$, $\{T,M\}_{a,E}$, and $\{T,M,P\}_{a,E,c}$ are
non-zero at $(a_0,E_0,c_0)\in\mathcal{D}$.  Then the functions
\begin{align*}
\phi_0&=\{T,u\}_{a,E}          ~~~~~~~~~~~~~~~~~   \psi_0=1\\
\phi_1&=\{T,M\}_{a,E}~u_x          ~~~~~~~~~~~~\psi_1=\int_0^x\phi_2(s)ds\\
\phi_2&=\{u,T,M\}_{a,E,c}         ~~~~~~~~~~~~\psi_2=\{T,M\}_{E,c}+\{T,M\}_{a,E}u
\end{align*}
satisfy
\begin{align*}
L_0\phi_0&=L_0\phi_1=0~~~~~~~~~~~~~~~~~~~~ L_0^\dag \psi_0=L_0^\dag\psi_2=0\\
L_0\phi_2&=-\phi_1 ~~~~~~~~~~~~~~~  ~~~~~~~~~~~L_0^\dag \psi_1=\psi_2
\end{align*}
In particular, if we further assume that $\{T,M\}_{a,E}$ and $\{T,M,P\}_{a,E,c}$ are non-zero at $(a_0,E_0,c_0)$,
then the functions $\{\phi_j\}_{j=1}^3$ forms a basis for the generalized null-eigenspace of $L_0$, and the
functions $\{\psi_j\}_{j=1}^3$ forms a basis for the generalized null-eigenspace of $L_0^\dag$.
\end{lemma}

Throughout the rest of our analysis, we will make the non-degeneracy assumption that the quantities $T_E$, $\{T,M\}_{a,E}$,
and $\{T,M,P\}_{a,E,c}$ are non-zero (see the previous section).  Lemma \ref{lem:Jordan} then implies, in essence,
that the elements of the $T$-periodic kernel of the unperturbed operator $L_0$ are given by elements of the tangent
space to the (two-dimensional) manifold of solutions \emph{of fixed period and fixed wavespeed}, while the element of the
first generalized kernel is given by a vector in the tangent space to the (three-dimensional) manifold of solutions \emph{of fixed period}
with no restrictions on wavespeed.  It immediately follows that the origin is a $T$-periodic eigenvalue
(corresponding to $\eps=0$, i.e. $\kappa=0$) of $L_0$ of algebraic multiplicity three and geometric multiplicity two.  Next,
we vary $\kappa$ in a neighborhood of zero to express the three eigenvalues bifurcating from the origin
and consider the spectral problem
\[
L_{\eps}v(\eps)=\mu(\eps)v(\eps)
\]
for $|\eps|\ll 1$, where we now make the additional assumption that the three branches
of the function $\mu(\eps)$ bifurcating from the $\mu(0)=0$ state are distinct\footnote{In the case where two or more
branches coincide to leading order, more delicate analysis is required than what is presented here.  In particular,
all asymptotic expansions must be continued to at least the next order in order to appropriately track
all three branches.}.
%
% of the form
%\[
%\mu_j(\kappa)=\lambda_j\kappa+o(\kappa),
%\]
%where we make the further assumption that the $\lambda_j$ are distinct.  The fact that the $T$-periodic
%null-space of the operator $L_0$ has a non-trivial Jordan structure suggests we must
%take extra care in our perturbation analysis.  To analyze the way in which this
%Jordan block breaks for small $\eps$, we expand the operator $L_{\eps}$ as
%\[
%L_{\eps}=L_0+\eps L_1+\eps^2L_2+\eps^3L_3
%\]
%and consider the spectral problem
%\[
%L_{\eps}v(\eps)=\mu(\eps)v(\eps)
%\]
%for $|\eps|\ll 1$.
Our first goal is to show that the spectrum $\mu(\eps)$ and hence the corresponding eigenfunctions $v(\eps)$ are
sufficiently smooth ($C^1$) in $\eps$.
The stronger result of analyticity was proved in \cite{BrJ} using the Weierstrass-Preparation theorem and the Fredholm alternative.
Here, we follow the methods from \cite{JZ5} to offer an alternative proof which is more suitable for our methods.

\begin{lemma}\label{lem:c1}
Assuming the quantities $\{T,M\}_{a,E}$ and $\{T,M,P\}_{a,E,c}$ are non-zero, the eigenvalues $\mu_j(\eps)$ of $L_{\eps}$
are $C^1$ functions of $\eps$ for $|\eps|\ll 1$.
\end{lemma}

\begin{proof}
To begin notice that since $\mu=0$ is an isolated eigenvalue of $L_0$, the associated total right eigenprojection $R(0)$ and total
left eigenprojection $L(0)$ perturb analytically in both $\eps$ (see \cite{K}).  It follows that we may find locally analytic right
and left bases $\{v_j(\eps)\}_{j=1}^3$ and $\{\tilde{v}_j(\eps)\}_{j=1}^3$ of the associated total eigenspaces given by the range of the projections
$R(\eps)$ and $L(\eps)$ such that $v_j(0)=\phi_j$ and $\tilde{v}_j(0)=\psi_j$.  Further defining the vectors
$V=(v_1,v_2,v_3)$ and $\tilde{V}=(\tilde{v}_1,\tilde{v}_2,\tilde{v}_3)^*$, where $*$ denotes the matrix adjoint, we may convert
the infinite-dimensional perturbation problem for the operator $L_{\eps}$ to a $3\times 3$ matrix perturbation problem for the matrix
\begin{equation}
M_{\eps}:=\left<\tilde{V}^*(\eps),L_{\eps}V(\eps)\right>_{L^2_{\rm per}([0,T])}.\label{eqn:meps}  %=M_0+\eps M_1+\eps^2M_2+\mathcal{O}(|\eps|^3).
\end{equation}
In particular, the eigenvalues of the matrix $M_{\eps}$ are coincide precisely with the eigenvalues
$\mu_j(\eps)$ of the operator $L_\eps$ lying in a neighborhood of $\mu=0$, and the associated left and right eigenfunctions
of $L_{\eps}$ are
\[
f_j=Vw_j \textrm{ and }\tilde{f}_j=\tilde{w}_j\tilde{V}^*
\]
where $w_j$ and $\tilde{w}_j$ are the associated right and left eigenvectors of $M_\eps$, respectively.  Thus, to demonstrate
the desired smoothness of the spectrum of the operator $L_{\eps}$  near the origin for $|\eps|\ll 1$, we need only demonstrate the
three eigenvalues of the matrix $M_{\eps}$ have the desired smoothness properties, which is a seemingly much easier task.

Next, we expand the vectors
\begin{align}\label{eqn:bvecexp}
v_j(\eps)&=\phi_j+\eps  q_j(\eps)+\mathcal{O}(|\eps|^2)\\
\tilde{v}_j(\eps)&=\psi_j+\eps  \tilde{q}_j(\eps)+\mathcal{O}(|\eps|^2)
\end{align}
and expand the matrix $M_{\eps}$ as
\[
M_{\eps}=M_0+\eps M_1+\mathcal{O}(|\eps|^2).
\]
Notice there is some flexibility in our choice of the functions $q_j$ and $\tilde{q}_j$, a fact
that we will exploit later.  Now, however, our calculation does not depend on a particular choice
of the expansions in \eqref{eqn:bvecexp}.
From Lemma \ref{lem:Jordan} a straight forward computation shows that
\begin{equation}
M_0=\left(
      \begin{array}{ccc}
        0 & 0 & 0 \\
        0 & 0 & \left<\psi_1,L_0\phi_2\right>  \\
        0 & 0 & 0 \\
      \end{array}
    \right)\label{eqn:m0}
\end{equation}
where $\left<\psi_1,L_0\phi_2\right>=\frac{1}{2}\{T,M\}_{a,E}\{T,M,P\}_{a,E,c}$ is non-zero by assumption reflecting the Jordan structure
of the unperturbed operator $L_0$.  By standard matrix perturbation theory, the spectrum of the matrix $M_{\eps}$ is $C^1$ in $\eps$
provided the entries $[M_1]_{1,2}$ and $[M_1]_{3,2}$ of the matrix $M_1$ are both zero.  Using Lemma \ref{lem:Jordan} again,
it indeed follows that
\begin{align*}
[M_1]_{1,2}&=\left<\psi_0,L_1\phi_1+L_0q_1\right>+\left<\tilde{q}_0,L_0\phi_1\right>=0\\
[M_1]_{3,2}&=\left<\psi_2,L_1\phi_1+L_0q_1\right>+\left<\tilde{q}_2,L_0\phi_1\right>=0
\end{align*}
and hence the spectrum is $C^1(\RM)$ in the parameter $\eps$ near $\eps=0$ as claimed.
\end{proof}

As a result of Lemma \ref{lem:c1}, the associated (non-normalized) eigenfunctions bifurcating from the generalized
null-space are $C^1$ in the parameter $\eps$ in a neighborhood of $\eps=0$.  Moreover, our overall strategy is now
clear: since the eigenvalues of $M_{\eps}$ correspond to the eigenvalues of the Bloch operator $L_{\eps}$ near the
origin, we need only study the characteristic polynomial of the matrix $M_{\eps}$ near $\eps=0$ in order
to understand the modulational stability of the underlying periodic wave.  However, notice by equation \eqref{eqn:m0}
the unperturbed matrix $M_0$ has a a non-trivial Jordan block, and hence the analysis of the bifurcating
eigenvalues must be handled with care.  In order to describe the breaking of the two-by-two Jordan block described in Lemma \ref{lem:Jordan}
for $\eps$ small, we rescale matrix $M_{\eps}$ in \eqref{eqn:meps} as
\[
\hat{M}_{\eps}=\eps^{-1}S(\eps)^{-1}M_{\eps}S(\eps)
\]
where
\[
S(\eps)=\left(
          \begin{array}{ccc}
            \eps & 0 & 0 \\
            0 & 1 & 0 \\
            0 & 0 & \eps \\
          \end{array}
        \right).
\]
In particular, the matrix $\hat{M}_{\eps}$ is an analytic matrix-valued function of $\eps$, and the eigenvalues
of $\hat{M}_{\eps}$ are given by $\nu_j(\eps)=\eps^{-1}\mu_j(\eps)$, where $\mu_j(\eps)$ represents the eigenvalues
of $M_{\eps}$.  Notice that the second coordinate of the vectors in $\CM^{n+1}$ in the perturbation problem \eqref{eqn:meps}
corresponds to the coefficient of $u_x$ to variations $\psi$ in displacement.  Thus, the above rescaling
amounts to substituting for $\psi$ the variable $|\eps|\psi\sim\psi_x$ of the Whitham average system.
Our next goal is to prove that the characteristic polynomial for the rescaled matrix $M_{\eps}$ agrees with that
of the linearized dispersion relation $\hat{\Delta}(\mu,\kappa)$ corresponding to the homogenized system \eqref{eqn:whitham}, i.e. we want to prove that
\[
\det\left(\eps\hat{M}_{\eps}-\mu \left<S(\eps)^{-1}\tilde{V}_{\eps}^*V_{\eps}S(\eps)\right>\right)=\hat{\Delta}(\mu,\kappa)+\mathcal{O}(|\mu|^4+|\kappa|^4)
\]
for some non-zero constant $C$.  To this end, we begin by determining the required variations in the vectors $\tilde{V}_{\eps}$ and $V_{\eps}$ near $\eps=0$ contribute
to leading order in the above equation.  We begin by studying the structure of the matrix $\hat{ M}_{\eps}$.

\begin{lemma}\label{lem:meps}
The matrix rescaled $\hat{M}_{\eps}$ can be expanded as
\[
\hat{M}_{\eps}=\eps^{-1}\left<S^{-1}(\eps)\left(
                                   \begin{array}{c}
                                     \psi_0 \\
                                     \psi_1 \\
                                     \psi_2+\eps \tilde{q}_2 \\
                                   \end{array}
                                 \right)L_{\eps}\left(\phi_0,\phi_1+\eps q_1,\phi_2\right)S(\eps)\right>_{L^2_{\rm per}([0,T])}+o(1)
\]
in a neighborhood of $\eps=0$.  In particular, the only first order variations of the vectors $V_{\eps}$ and $\tilde{V}_{\eps}$
which contribute to leading order are $\tilde{q}_2$ and $q_1$.
\end{lemma}

\begin{proof}
The idea is to undo the rescaling and find which entries of the unscaled matrix $M_{\eps}$ contribute to leading order.
To begin, we expand the non-rescaled matrix $M_{\eps}$ from \eqref{eqn:meps} as
\[
M_{\eps}=M_0+\eps M_1+\eps^2M_2+\mathcal{O}(|\eps|^3)
\]
and notice that $M_0$ was computed in \eqref{eqn:m0} and was shown to be nilpotent but non-zero, possessing a nontrivial
associated Jordan chain of height two.  Using Lemma \ref{lem:Jordan} a straight forward computation shows the matrix $M_1$
can be expressed as
\begin{align*}
M_1&=\left(
      \begin{array}{ccc}
        \left<\psi_0,L_1\phi_0+L_0q_0\right> & 0 & \left<\psi_0,L_1\phi_2\right>+\left<\tilde{q}_0,L_0\phi_2\right> \\
        * & \left<\psi_1,L_1\phi_1+L_0q_1\right> & * \\
        \left<\psi_2,L_1\phi_0+L_0q_0\right>  & 0 & \left<\psi_2,L_1\phi_2\right>+\left<\tilde{q}_2,L_0\phi_2\right> \\
      \end{array}
    \right)\\
&=
\left(
      \begin{array}{ccc}
        \left<\psi_0,L_1\phi_0\right> & 0 & \left<\psi_0,L_1\phi_2\right>+\left<\tilde{q}_0,L_0\phi_2\right> \\
        * & \left<\psi_1,L_1\phi_1+L_0q_1\right> & * \\
        \left<\psi_2,L_1\phi_0\right>  & 0 & \left<\psi_2,L_1\phi_2\right>+\left<\tilde{q}_2,L_0\phi_2\right> \\
      \end{array}
    \right)
\end{align*}
where the $*$ terms are not necessary as they contribute to higher order terms in the rescaled matrix $\hat{M}_{\eps}$.
Similarly, the relevant entries of the matrix $M_2$ are given by
\[
M_2=\left(
      \begin{array}{ccc}
        * & \left<\psi_0,L_2\phi_1+L_1q_1\right>+\left<\tilde{q}_0,L_1\phi_1+L_0q_1\right> & * \\
        * & * & * \\
        * & \left<\psi_2,L_2\phi_1+L_1q_1\right>+\left<\tilde{q}_2,L_1\phi_1+L_0q_1\right> & * \\
      \end{array}
    \right).
\]
Therefore, it follows the relevant entries of the matrix $M_{\eps}$ up to $\mathcal{O}(|\eps|^2)$ can be evaluated
using only the variations $q_1$ and $\tilde{q}_2$, as claimed.
\end{proof}

The main point of the above lemma is that in order to compute to projection of the operator $L_{\eps}$ onto the eigenspace
bifurcating null-space at $\eps=0$ to leading order, you only need to consider the variations in the bottom of the left and right Jordan chains: all
other variations contribute to terms of higher order.  Our next lemma shows the variation in the $\psi_0$ direction is also needed
to compute the corresponding projection of the identity to leading order.

\begin{lemma}\label{lem:ieps}
Define the matrix $\tilde{ I}_{\eps}:=\left<S(\eps)^{-1}\tilde{V}^*_{\eps}\tilde{V}_{\eps}S(\eps)\right>$.  Then $\tilde{ I}_{\eps}$
can be expanded near $\eps=0$ as
\[
\tilde{I}_{\eps}=\left<S^{-1}(\eps)\left(
                                   \begin{array}{c}
                                     \psi_0+\eps\tilde{q}_0 \\
                                     \psi_1 \\
                                     \psi_2+\eps \tilde{q}_2 \\
                                   \end{array}
                                 \right)\left(\phi_0,\phi_1+\eps q_1,\phi_2\right)S(\eps)\right>_{L^2_{\rm per}([0,T])}+o(1).
\]
\end{lemma}

\begin{proof}
As in the proof of Lemma \ref{lem:meps}, we undo the rescaling and find the terms which contribute to leading order.
To begin, we expand the matrix $I_{\eps}:=\left<\tilde{V}_{\eps}^*V_{\eps}\right>$ as
\[
I_{\eps}=I_0+\eps I_1+\mathcal{O}(|\eps|^2)
\]
Using the fact that $\left<\psi_i,\phi_j\right>=0$ if $i\neq j$, it follows that
\[
 I_0=\left(
                                          \begin{array}{ccc}
                                            \left<\psi_0,\phi_0\right> & 0 & 0 \\
                                            0 & \left<\psi_1,\phi_1\right> & 0 \\
                                            0 & 0 & \left<\psi_2,\phi_2\right> \\
                                          \end{array}
                                        \right)
\]
and
\[
I_1=\left(
            \begin{array}{ccc}
              * & \left<\psi_0,q_1\right>+\left<\tilde{q}_0,\phi_1\right> & * \\
              * & * & * \\
              * & \left<\psi_2,q_1\right>+\left<\tilde{q}_2,\phi_1\right> & * \\
            \end{array}
          \right)
\]
from which the lemma follows by rescaling.
\end{proof}

With the above preparations, we can now project the operator $L_{\eps}-\mu(\eps)$ onto the total eigenspace bifurcating from
the origin for small $|\eps|\ll 1$.  Our claim is that projected and rescaled matrix agrees with the Whitham system, up to
a similarity transformation.  To show this, we begin by showing the characteristic polynomial of the corresponding rescaled
matrix agrees with the linearized dispersion relation $\hat{\Delta}(\mu,\kappa)$ corresponding to the linearized
Whitham system.  Once this is established, the fact that the bifurcating eigenvalues are distinct will imply the desired similarity by the analysis in \cite{JZ1}.

In order to compute the characteristic polynomial of the matrix projection of the operator $L_{\eps}-\mu(\eps)$, we must
now make a few specific choices as to the variations in the vectors $\tilde{V}_{\eps}$ and $V_{\eps}$.
As noted in the proof of Lemma \ref{lem:c1},there is quite a bit of flexibility in our choice of the expansion \eqref{eqn:bvecexp}.
As a naive choice, we could require that $v_j(\eps)$ be an eigenfunction of $L_{\eps}$ for small $\eps$.  For example,
we may choose $v_0$ and $v_2$ to satisfy
\[
L_0q_j=(\lambda_1-L_1)\phi_j,~~j=0,2
\]
where $\lambda_1$ is the corresponding eigenvalue bifurcating from the origin.  Similarly, we may choose $\tilde{v}_0$ and $\tilde{v_1}$
to satisfy
\[
L_0^\dag\tilde{q}_j=(\lambda_1^*-L_1)\psi_j,~~j=0,1
\]
where $*$ denotes complex conjugation.  However, this is not necessary: we only need $v_j(\eps)$ to provide a \textit{basis} for the null-eigenspace of the operator $L_{\eps}$.
To illustrate this, instead of choosing the variation in the $\phi_1$ direction to satisfy an eigenvalue equation
we notice there are two distinct eigenvalues with expansions
\begin{align*}
\mu_1(\eps)&=\eps \lambda_1+o(|\eps|),\\
\tilde{\mu}_1(\eps)&=\eps\tilde{\lambda}_1+o(|\eps|)
\end{align*}
corresponding the the eigenfunctions $\phi_1+\eps g_1+o(|\eps|)$ and $\phi_1+\eps\tilde{g}_1+o(|\eps|)$.  In particular,
it follows that the functions $g_1$ and $\tilde{g}_1$ satisfy the equations
\begin{align*}
L_0g_1&=(\lambda_1-L_1)\phi_1,\\
L_0\tilde{g}_1&=(\tilde{\lambda}_1-L_1)\phi_1.
\end{align*}
Defining $q_1=\left(\tilde{\lambda}_1-\lambda_1\right)^{-1}\left(\tilde{\lambda}_1g_1-\lambda_1\tilde{g}_1\right)$ then, it follows
that the function $q_1$ to satisfies
\begin{align*}
L_0q_1&=\left(\tilde{\lambda}_1-\lambda_1\right)^{-1}\left(\tilde{\lambda}_1(\lambda_1-L_1)\phi_1-\lambda_1(\tilde{\lambda}_1-L_1)\phi_1\right)\\
&=-L_1\phi_1
\end{align*}
and hence may be chosen to satisfy $q_1\perp\textrm{span}\{\psi_0,\psi_1\}$.  To find a closed form expression for $q_1$,
notice that
\[
L_0\left(xu_x\right)=2u_{xxx}=-L_1u_x.
\]
Moreover, a direct calculation shows that the function $u_E$ satisfies $\mathcal{L}[u]u_E=0$ and the function
\[
\tilde{\phi}=-\{T,M\}_{a,E}\left(xu_x+\frac{T}{T_E}u_E\right)
\]
is $T$-periodic.  In particular, $L_0\tilde{\phi}=-L_1\phi_1$ and $\tilde{\phi}\perp\textrm{span}\{\psi_0,\psi_1\}$
and hence we may choose $q_1=\tilde{\phi}$.  Note there are many other choices for $q_1$ that are possible: our choice
is made to simplify the forthcoming calculations.  Similarly, since $\psi_2$ is at the bottom of the
Jordan chain of the null-space of $L_{0}^\dag$, we may choose $\tilde{g}_1$ to satisfy the equation
\[
L_0^\dag\tilde{q}_2=-L_1^\dag\psi_2.
\]
Unlike the variation in $\phi_1$, we do not need a closed form expression for $\tilde{g}_1$: the above defining
relation will be sufficient for our purposes.

With the above choices, all the necessary inner products described in Lemmas \ref{lem:meps} and \ref{lem:ieps}
may be evaluated explicitly.  Indeed, straight forward computations show that
\begin{align*}
M_0&=\left(
      \begin{array}{ccc}
        0 & 0 & 0 \\
        0 & 0 & \frac{1}{2}\{T,M\}_{a,E}\{T,M,P\}_{a,E,c} \\
        0 & 0 & 0 \\
      \end{array}
    \right)\\
M_1&=\left(
       \begin{array}{ccc}
         T_ET & * & 0 \\
         * & 0 & * \\
         T(T_E\{T,M\}_{E,c}+Ta\{T,M\}_{a,E}) & * & 0 \\
       \end{array}
     \right)\\
M_2&=\left(
       \begin{array}{ccc}
         * & T\{T,K\}_{a,E} & * \\
         * & * & * \\
         * & T\{T,M\}_{E,c}\{T,K\}_{a,E}+\frac{T\{T,M\}_{a,E}}{T_E}\left(T_a\{T,K\}_{a,E}-T\{T,M\}_{a,E}\right) & * \\
       \end{array}
     \right).
\end{align*}
and
\begin{align*}
I_0&=\left(
       \begin{array}{ccc}
         \{T,M\}_{a,E} & 0 & 0 \\
         0 & -\frac{1}{2}\{T,M\}_{a,E}\{T,M,P\}_{a,E,c} & 0 \\
         0 & 0 & \frac{1}{2}\{T,M\}_{a,E}\{T,M,P\}_{a,E,c} \\
       \end{array}
     \right)\\
I_1&=\left(
       \begin{array}{ccc}
         * & -\{T,M\}_{E,c}T-\{T,M\}_{a,E}M & * \\
         * & * & * \\
         * & 2\{T,M\}_{a,E}\{K,T,M\}_{a,E,c}-\{T,M\}_{E,c}^2T-2\{T,M\}_{E,c}\{T,M\}_{a,E}M-\{T,M\}_{a,E}^2P & * \\
       \end{array}
     \right).
\end{align*}
We can thus explicitly compute the rescaled matrices $\hat{M}_\eps$ and $\tilde{I}_{\eps}$ in terms of the underlying
solution $u$, which yields the following theorem.

\begin{thm}\label{thm:main1}
Let $(a_0,E_0,c_0)\in\mathcal{D}$ and assume the matrices $\frac{\partial(M\omega,P\omega,\omega)}{\partial(\dot{u})}$ and
$\frac{\partial(\dot{u}}{\partial(a,E,c)}$ are invertible at $(a_0,E_0,c_0)$.  Then the linearized dispersion relation
$\hat{\Delta}(\mu,\kappa)$ in \eqref{eqn:wpoly1} satisfies
\[
\det\left(\eps\hat{M}_{\eps}-\mu \left<S(\eps)^{-1}\tilde{V}_{\eps}^*V_{\eps}S(\eps)\right>\right)=C\hat{\Delta}(\mu,\kappa)+\mathcal{O}(|\mu|^4+|\kappa|^4)
\]
for some constant $C\neq 0$.  That is, up to a constant the linearized dispersion relation for the homogenized system \eqref{eqn:whitham}
accurately describes the low-frequency behavior of the spectrum of the Bloch-operator $L_{\xi}$.
\end{thm}

\begin{proof}
A straightforward computation using the above identities implies
\begin{align*}
\det\left(\eps\hat{M}_{\eps}-\mu \left<S(\eps)^{-1}\tilde{V}_{\eps}^*V_{\eps}S(\eps)\right>\right)
       &=\frac{T^3\{T,M\}_{a,E}^3}{2T_E}\left(\{T,M\}_{E,a}+T_a^2-2T_cT_E\right)\\
&+C\hat{\Delta}(\mu,\kappa)+\mathcal{O}(|\mu|^4+|\kappa|^4)
\end{align*}
for some non-zero constant $C=C(a,E,c)$.  Moreover, using the identity
\[
2T_c=P_E=M_a=-\frac{\sqrt{2}}{4}\oint_{\Gamma}\frac{u^2~du}{\left(E-V(u;a,c)\right)^{3/2}},
\]
which immediately follow from the integral formulas \eqref{period}-\eqref{momentum},
along with the fact that $T_a=M_E$ by \eqref{eqn:gradk}, it follows that
\[
T_a^2-2T_cT_E=T_aM_E-T_EM_a=\{T,M\}_{a,E}=-\{T,M\}_{E,a}
\]
and hence
\[
\det\left(\eps\hat{M}_{\eps}-\mu \left<S(\eps)^{-1}\tilde{V}_{\eps}^*V_{\eps}S(\eps)\right>\right)=C\hat{\Delta}(\mu,\kappa)+\mathcal{O}(|\mu|^4+|\kappa|^4)
\]
as claimed.
\end{proof}

Theorem \ref{thm:main1} provides a rigorous verification of the Whitham modulation equations for the gKdV equations.  While
this has recently been established in \cite{JZ1}, the important observation here is that the verification was independent
of the restrictive Evans function techniques.  Thus, the above computation can be used as a blue print of how to rigorously
justify Whitham expansions in more complicated settings where the Evans function framework is not available.  As a consequence
of our assumption that the eigenvalues of $M_{\eps}$ be distinct for $0<|\eps|\ll 1$, it follows that there must exist a similarity
transformation between the matrix
\[
\partial_{\eps}\left(\eps\hat{M}_{\eps}-\mu \left<S(\eps)^{-1}\tilde{V}_{\eps}^*V_{\eps}S(\eps)\right>\right)\big{|}_{\eps=0}
\]
and the matrix arising from the linearized Whitham system \eqref{eqn:wpoly1}.  Thus, the variations predicted by Whitham
to control the long-wavelength stability of a periodic traveling wave solution of the gKdV are indeed
the variations needed at the level of the linearized Bloch-expansion.

\section{Conclusions}

In this paper we considered the spectral stability of a periodic traveling wave of the gKdV equation to long-wavelength
perturbations.  Recently, this notion of stability has been the focus of much work in the context of viscous systems of conservation laws
\cite{OZ1,OZ3,OZ4,Se1,JZ4,JZ5} and nonlinear dispersive equations \cite{BrJ,BrJK,J2,JZ}.  While much of this work has utilized
the now familiar and powerful Evans function techniques, our approach of using a direct Bloch-decomposition
of the linearized operator serves to provide an elementary method in the one-dimensional setting considered here
as well as a robust method which applies in the more complicated setting of multi-periodic structures.  As such, our hope
is that the simple and straightforward analysis in this paper will be used as a blue print of how to justify the Whitham
modulation equations in settings which do not admit a readily computable Evans function.  In future work, we will report
on the application of this method to the long-wavelength stability of doubly periodic traveling waves of viscous
systems of conservation laws.

As a future direction of investigation, we note that Theorem \ref{thm:main1} and the assumption that the branches of spectrum
bifurcating from the origin are distinct suggests a possibly more algorithmic approach
to justifying the Whitham modulation equations via the above Bloch-expansion methods which furthermore does not
rely on comparing the linearized dispersion relations as in Theorem \ref{thm:main1}: it should be possible to
justify the Whitham expansion by direct comparison of inner products and showing the mentioned similarity
directly.  As a first step, we suggest beginning this line of investigation by studying the gKdV equation
as in the current paper and \cite{JZ1} and expressing the full Whitham system in terms of inner products.
Demonstrating the desired similarity in this way, while not completely necessary, could substantially simplify the amount
of work required to justify the Whitham equations by not requiring an initial justification at the spectral level
(a possibly daunting task in more general situations).


\begin{thebibliography}{J4}

{\footnotesize

\bibitem[B]{B} T. B. Benjamin, \emph{The stability of solitary waves}.  Proc. Roy. Soc. (London) Ser. A,   328, (1972).

\bibitem[BD]{BD} N. Bottman and B. Deconinck, \emph{KdV cnoidal waves are linearly stable}, DCDS-A 25, (2009).

\bibitem[BrJ]{BrJ} J. C. Bronski and M. Johnson, \emph{The modulational instability for a generalized Korteweg-de Vries equation},
Arch. Ration. Mech. Anal., DOI 10.1007/s00205-009-0270-5, (2009).


\bibitem[BrJK]{BrJK}  J. C. Bronski, M. Johnson, and T. Kapitula, \emph{An index theorm for the stability of periodic traveling
waves of KdV type}, preprint.


\bibitem[BM]{BM} E. Barthelemy and H. Michallet, \emph{Experimental study of large interfacial solitary waves}, Journal of Fluid Mechanics, 336 (1998).

\bibitem[BaMo]{BaMo} J. F. Bampi and A. Morro, \emph{Korteweg-de Vries equation and nonlinear waves}, Lettere Al Nuovo Ci-
mento, 26(2), (1979).

\bibitem[Bo]{Bo} J. L. Bona, \emph{On the stability theory of solitary waves}, Proc. Roy. Soc. London Ser. A, 344 (1975).

\bibitem[BO]{BO} T. L. Burch and A. R. Osborne, \emph{Internal Solitons and the Andaman sea}, Science, 208 (1940), no. 4443.


\bibitem[DK]{DK} B. Deconinck and T. Kapitula, \emph{On the orbital (in)stability of spatially periodic stationary solutions
of generalized Korteweg-de Vries equations}, preprint.

\bibitem[GH]{GH} Th. Galley  and M. H\u{a}r\u{a}g\c{u}s, \emph{Stability of small periodic waves for the nonlinear Schrödinger equation},
J. Diff. Equations, 234, (2007).


\bibitem[Ha]{Ha} M. H\u{a}r\u{a}g\c{u}s, \emph{Stability of  periodic waves for the generalized BBM equation},
Rev. Roumaine Maths. Pures Appl., 53, (2008).

\bibitem[HK]{HK} M. H\u{a}r\u{a}g\c{u}s and T. Kapitula, \emph{On the spectra of periodic waves for infinite-dimensional Hamiltonian systems},
Physica D, 237(20), (2008)

\bibitem[HaLS]{HaLS} M. H\u{a}r\u{a}g\c{u}s, E. Lombardi, and A. Scheel, \emph{Stability of wave trains in the Kawahara equation},
J. Math. Fluid Mech., 8, (2006).


\bibitem[KD]{KD} B. C. Kalita and R. Das, \emph{Modified Korteweg de Vries (mkdv) and Korteweg-de Vries (kdv) solitons in a warm plasma with negative ions and electrons'
drift motion}, Journal of the Physical Socity of Japan, 71 (2002), no. 21.

\bibitem
[J1]{J1} M. Johnson, {\it Nonlinear Stability of Periodic Traveling Wave Solutions of the Generalized Korteweg-de Vries Equation},
SIMA, 41 no. 5 (2009).

\bibitem[J2]{J2} M. Johnson, {\it On the Stability of Periodic Solutions of the Generalized Benjamin-Bona-Mahony Equation}, submitted, (2009).

\bibitem[J3]{J3} M. Johnson, {\it The Transverse Instability of Periodic Waves in Zakharov-Kuznetsov Type Equations},
Studies in Applied Mathematics, DOI: 10.1111/j.1467-9590.2009.00473.x, (2009).

\bibitem[JZ]{JZ} M. Johnson and K. Zumbrun, {\it Transverse Instability of Periodic Traveling Waves in the Generalized Kadomtsev-Petviashvili Equation},
submitted, (2009).

\bibitem[JZ1]{JZ1} M. Johnson and K. Zumbrun, {\it Rigorous Justification of the Whitham Modulation Equations for the Generalized Korteweg-de Vries Equation}, to appear in Studies in Applied Mathematics.

\bibitem[JZ4]{JZ4} M. Johnson and K. Zumbrun, {\it Nonlinear stability of periodic traveling wave solutions of systems of conservation laws in dimensions one and two},
submitted, (2009).

\bibitem
[JZ5]{JZ5} M. Johnson and K. Zumbrun,
{\it Nonlinear stability of periodic traveling wave solutions
of systems of viscous conservation laws in the generic case},
submitted, (2010).

\bibitem[M]{M} A. A. Mamun, \emph{Nonlinear propagation of ion-acoustic waves in a hot magnetized plasma with vortex-like electron distribution},
Physics of Plasmas, 5 (1998), no. 1.

\bibitem[MS]{MS} A. Mushtaq and H. A. Shah, \emph{Study of non-maxwellian trapped electrons by using generalized (r,q)
distribution function and their effects on the dynamics of ion acoustic solitary wave}, Physics of Plasmas, 13 (2006).

\bibitem
[K]{K} T. Kato,
{\it Perturbation theory for linear operators},
Springer--Verlag, Berlin Heidelberg (1985).

\bibitem[KdV]{KdV} D. J. Korteweg and G. de Vries, \emph{On the change of form of long waves advancing in a rectangular channel and on a new
type of stationary waves}, Philosophical Magazine, 36(1895).

\bibitem[PW]{PW} R. L. Pego and M. I. Weinstein, \emph{Eigenvalues and instabilities of solitary waves}, Philos. Trans. Roy. Soc. London Ser. A., 340, (1992).

\bibitem
[OZ1]{OZ1}
M. Oh and K. Zumbrun, \textit{Stability of periodic
solutions of viscous conservation laws with viscosity-
1. Analysis of the Evans function},
Arch. Ration. Mech. Anal. 166 (2003), no. 2, 99--166.

\bibitem
[OZ3]{OZ3} M. Oh, and K. Zumbrun,
\textit{Low-frequency stability analysis of periodic
traveling-wave solutions of viscous conservation laws
in several dimensions},
Journal for Analysis and its Applications, 25 (2006), 1--21.

\bibitem
[OZ4]{OZ4} M. Oh, and K. Zumbrun,
\textit{Stability and asymptotic behavior of
traveling-wave solutions of viscous conservation laws
in several dimensions}, to appear, Arch. Ration. Mech. Anal.
%TODO: update? errata?

\bibitem
[Se1]{Se1} D. Serre,
{\it Spectral stability of periodic solutions of viscous conservation laws:
Large wavelength analysis}, Comm. Partial Differential Equations 30 (2005),
no. 1-3, 259--282.

\bibitem[SH]{SH} H. Segur and J. Hammack, \emph{Soliton models of long internal waves}, Journal of Fluid Mechanics, 118 (1982).

}
\end{thebibliography}
\end{document}